\documentclass[11pt]{article}

\usepackage{latexsym,amssymb,amsmath,amsfonts,latexsym,amscd,amsthm}

\title{\textbf{Contact metric $\left(\kappa,\mu\right)$-spaces as\\ bi-Legendrian manifolds}}
 \author{\begin{tabular}{cc}
  \textsc{Beniamino Cappelletti Montano, Luigia Di Terlizzi}\\
   Department of Mathematics, University of Bari \\
  Via E. Orabona, 4 \\
  I-70125 Bari (Italy) \\
  \textsf{cappelletti@dm.uniba.it, terlizzi@dm.uniba.it}
    \end{tabular}}

\newtheorem{lem}{Lemma}[section]
\newtheorem{cor}{Corollary}[section]
\newtheorem{theo}{Theorem}[section]

\newtheorem{exam}{Example}[section]
\newtheorem{rem}{Remark}[section]
\newtheorem{prop}{Proposition}[section]
\numberwithin{equation}{section}

\date{}

\begin{document}

\maketitle

\begin{abstract}
We regard a contact metric manifold whose Reeb vector field belongs to the $\left(\kappa,\mu\right)$-nullity distribution as a bi-Legendrian manifold and we study its canonical bi-Legendrian structure. Then we characterize contact metric $\left(\kappa,\mu\right)$-spaces in terms of a canonical connection which can be naturally defined on them.
\end{abstract}

\section{Introduction}

Contact metric $\left(\kappa,\mu\right)$-spaces, introduced in
\cite{blair1} by D. E. Blair, T. Kouforgiorgos and B. J.
Papantoniou, are those contact metric manifolds
$\left(M,\phi,\xi,\eta,g\right)$ for which the Reeb vector field
$\xi$ belongs to the $\left(\kappa,\mu\right)$-nullity
distribution, i.e. satisfies, for all vector fields $V$ and $W$ on
$M$,
\begin{equation}\label{definizione}
R_{V
W}\xi=\kappa\left(\eta\left(W\right)V-\eta\left(V\right)W\right)+\mu\left(\eta\left(W\right)hV-\eta\left(V\right)hW\right),
\end{equation}
for some real numbers $\kappa$ and $\mu$, where $2h$ is the Lie
derivative of $\phi$ in the direction of $\xi$. This definition
can be regarded as a generalization both of the Sasakian condition
$R_{V W}\xi=\eta\left(W\right)V-\eta\left(V\right)W$ and of those
contact metric manifolds verifying $R_{V W}\xi=0$ which were
studied by D. E. Blair in \cite{blair-1}.

Recently contact metric $\left(\kappa,\mu\right)$-spaces have been
studied by various authors (\cite{boeckx0}, \cite{boeckx1},
\cite{boeckx2}, \cite{cho}, \cite{kouforgiorgos}, etc.) and
several important properties of these manifolds have been
discovered. In fact there are many motivations for studying
$\left(\kappa,\mu\right)$-spaces: the first is that, in the
non-Sasakian case (that is for $\kappa\neq 1$), the condition
\eqref{definizione} determines the curvature completely; moreover,
while the values of $\kappa$ and $\mu$ change, the form of
\eqref{definizione} is invariant under $\cal D$-homothetic
deformations; finally, there are non-trivial examples of these
manifolds, the most important being the unit tangent sphere bundle
of a Riemannian manifold of constant sectional curvature with the usual contact metric
structure.

A complete classification of contact metric
$\left(\kappa,\mu\right)$-spaces has been given in \cite{boeckx1}
by E. Boeckx,  who proved
also that any non-Sasakian contact metric $(\kappa,\mu)$-space is
locally homogeneous and strongly locally $\phi$-symmetric
(\cite{boeckx0}).

One of the peculiarities of these manifolds is that they give rise
to three mutually orthogonal distributions ${\cal D}_{\lambda}$,
${\cal D}_{-\lambda}$ and $\mathbb{R}\xi$, corresponding to the
eigenspaces of the operator $h$. In particular ${\cal
D}_{\lambda}$ and ${\cal D}_{-\lambda}$ define two transverse
Legendrian foliations of $M$ so that  these manifolds are endowed
with a bi-Legendrian structure.

In the same years the theory of Legendrian foliations has been
developed by M. Y. Pang, P. Libermann and N. Jayne (cf.
\cite{pang}, \cite{libermann}, \cite{jayne1}), so it seems to be
tempting to use the techniques and the language of Legendrian
foliations for the study of contact metric
$\left(\kappa,\mu\right)$-spaces and to begin the investigation of
the interactions between these two areas of the contact geometry.
This is what we set out to do in this article.

The paper is organized as follows. After some preliminaries on
contact metric manifolds and Legendrian foliations, in $\S$ 3 we
study the Legendrian foliations canonically defined in any contact
metric $\left(\kappa,\mu\right)$-space. We find, for both the
foliations, an explicit formula of the invariant $\Pi$ introduced
by Pang for classifying Legendrian foliations (cf. \cite{pang})
and we see that the Legendrian foliations in question are,
according to this classification, either non-degenerate or flat.
Then we relate these invariants to the invariant $I_M$ used by
Boeckx in \cite{boeckx1} for classify contact metric
$\left(\kappa,\mu\right)$-spaces. In $\S$ 4 we attach to any
contact metric $\left(\kappa,\mu\right)$-space a linear connection
in a canonical way. We study the properties of this connection
and, using it, we give an interpretation of the notion of contact
metric $\left(\kappa,\mu\right)$-space in terms of bi-Legendrian
structures. In particular, we prove the following characterization
of contact metric $\left(\kappa,\mu\right)$-spaces.

\begin{theo}\label{principale}
A contact metric manifold  $\left(M,\phi,\xi,\eta,g\right)$ is a
contact metric $\left(\kappa,\mu\right)$-space if and only if $M$
admits an orthogonal bi-Legendrian structure $\left({\cal F},
{\cal G}\right)$ such that the corresponding bi-Legendrian
connection $\bar\nabla$ satisfies $\bar{\nabla}\phi=0$ and
$\bar{\nabla} h=0$. Furthermore, the bi-Legendrian structure
$\left({\cal F}, {\cal G}\right)$ coincides with that one
determined by the eigenspaces of $h$.
\end{theo}

This theorem should be compared with the well-known results,
obtained by N. Tanaka (cf. \cite{tanaka}) and, independently,  S.
M. Webster (\cite{webster}). They proved that
 any strongly pseudo-convex CR-manifold admits a unique
linear connection $\tilde{\nabla}$ such that the tensors $\phi$,
$\eta$, $g$ are all $\tilde{\nabla}$-parallel and whose torsion
satisfies $\tilde{T}\left(Z,Z'\right)=2\Phi\left(Z,Z'\right)\xi$
for all $Z,Z'\in\Gamma\left(\cal D\right)$ and
$\tilde{T}\left(\xi,\phi V\right)=-\phi
\tilde{T}\left(\xi,V\right)$ for all $V\in\Gamma\left(TM\right)$.
In view of this remark and the mentioned theorem of Boeckx that
any contact metric $\left(\kappa,\mu\right)$-space is a strongly
pseudo-convex CR-manifold, one can see that the connection
mentioned in Theorem \ref{principale} plays the same role for
contact metric $\left(\kappa,\mu\right)$-space that the
Tanaka-Webster connection has for CR-manifolds. As we shall see,
the connection $\bar{\nabla}$ uniquely determines a contact metric
$\left(\kappa,\mu\right)$-space modulo $\cal D$-homothetic
deformations and it reveals very useful in the study of this kind
of contact metric manifolds.

\section{Preliminaries}

\subsection{Contact manifolds}

An \emph{almost contact metric manifold} is a
$\left(2n+1\right)$-dimensional Riemannian manifold
$\left(M,g\right)$ which admits a tensor field $\phi$ of type
$(1,1)$, a global $1$-form $\eta$ and a global vector field $\xi$,
called \emph{Reeb vector field}, satisfying
\begin{equation}\label{contattometrica}
\eta\left(\xi\right)=1, \ \ \phi^2 V = -V+\eta\left(V\right)\xi, \ \
g\left(\phi V,\phi
W\right)=g\left(V,W\right)-\eta\left(V\right)\eta\left(W\right),
\end{equation}
for all vector fields $V$ and $W$ on $M$. Given an almost contact
metric manifold one can define a $2$-form $\Phi$, called the
\emph{fundamental  $2$-form} of the structure, by
$\Phi\left(V,W\right)=g\left(V,\phi W\right)$. Then we say that
$\left(M,\phi,\xi,\eta,g\right)$ is a \emph{contact metric
manifold} if the additional property $d\eta=\Phi$ holds. From
\eqref{contattometrica} it can be proven that (cf. \cite{blair0})
\begin{description}
    \item[(i)] $\phi\xi=0$, $\eta\circ\phi=0$,
    \item[(ii)] $\nabla_{\xi}\phi=0$ and $\nabla_{\xi}\xi=0$,
    \item[(iii)] $\phi|_{\cal D}$ is an isomorphism,
\end{description}
where $\nabla$ denotes the Levi Civita connection and  $\cal
D=\ker\left(\eta\right)$ is the $2n$-dimensional distribution
orthogonal to $\xi$ and called the \emph{contact distribution}. It is also easy to prove that for any $X\in\Gamma\left(\cal D\right)$ the bracket $\left[X,\xi\right]$ still belongs to $\cal D$.

In any contact metric manifold the $1$-from $\eta$ satisfies the
relation
\begin{equation}\label{contactform}
\eta\wedge\left(d\eta\right)^n\neq 0
\end{equation}
everywhere on $M$. Any $\left(2n+1\right)$-dimensional smooth
manifold which carries a global $1$-form satisfying
\eqref{contactform} is called a \emph{contact manifold}. Thus any
contact metric manifold is a contact manifold. Conversely, it is
well-known that any contact manifold admits a compatible contact
metric structure $\left(\phi,\xi,\eta,g\right)$. It should be
remarked that \eqref{contactform} implies that the contact
distribution $\cal D$ is never integrable.

Given a contact metric manifold, we can define a tensor field $h$
by $h=\frac{1}{2}{\cal L}_{\xi}\phi$, $\cal L$ denoting the Lie
differentiation. It can be shown (cf. \cite{blair0}) that $h$ is a
trace-free, symmetric operator verifying $h\xi=0$, $\phi h=-h\phi$
and
\begin{equation}\label{Kcontatto}
\nabla_{V}\xi=-\phi h V - \phi V
\end{equation}
for all $V\in\Gamma\left(TM\right)$. Moreover $\xi$ is Killing
 if and only if $h$ vanishes identically; in this case
we say that $\left(M,\phi,\xi,\eta,g\right)$ is a
\emph{$K$-contact manifold}.

On a contact metric manifold $M$ one can define an almost complex
structure $J$ on the product manifold $M\times\mathbb{R}$ by
setting $J\left(V,f\frac{d}{dt}\right)=\left(\phi V - f\xi,
\eta\left(V\right)\frac{d}{dt}\right)$, where $V$ is a vector
field tangent to $M$ and $f$ a function on $M\times\mathbb{R}$. If
the almost complex structure $J$ is integrable then
$\left(M,\phi,\xi,\eta,g\right)$ is said to be \emph{Sasakian}. It
is well-known that each of the following conditions characterizes
Sasakian manifolds
\begin{gather}
\left(\nabla_{V}\phi\right)W=g\left(V,W\right)\xi-\eta\left(W\right)V\\
R_{V W}\xi=\eta\left(W\right)V-\eta\left(V\right)W
\label{curvaturasasaki}
\end{gather}
for all vector fields $V$ and $W$ on $M$. A generalization of the
condition \eqref{curvaturasasaki} leads to the notion of
$\left(\kappa,\mu\right)$-manifold. If the curvature tensor field
of a contact metric manifold satisfies \eqref{definizione} for
some real numbers $\kappa$ and $\mu$ we say that $\xi$ belongs to
the
 $\left(\kappa,\mu\right)$-nullity distribution or, simply, that
$\left(M,\phi,\xi,\eta,g\right)$ is a contact metric
$\left(\kappa,\mu\right)$-space. This manifolds were introduced
and deeply studied in \cite{blair1}. Among other things, the
authors proved the following results.

\begin{theo}[\cite{blair1}]\label{teoremagreci}
Let $\left(M,\phi,\xi,\eta,g\right)$ be a contact metric manifold
with $\xi$ belonging to the $\left(\kappa,\mu\right)$-nullity
distribution. Then $\kappa\leq 1$. Moreover, if $\kappa=1$ then
$h=0$ and $\left(M,\phi,\xi,\eta,g\right)$ is a Sasakian manifold;
if $\kappa<1$, the contact metric structure is not Sasakian and
$M$ admits three mutually orthogonal integrable distributions
${\cal D}_0=\mathbb{R}\xi$, ${\cal D}_{\lambda}$ and ${\cal
D}_{-\lambda}$ corresponding to the eigenspaces of $h$, where
$\lambda=\sqrt{1-\kappa}$.
\end{theo}

\begin{theo}[\cite{blair1}]
Let $\left(M,\phi,\xi,\eta,g\right)$ be a contact metric manifold
with $\xi$ belonging to the $\left(\kappa,\mu\right)$-nullity
distribution. Then the following relation hold, for any
$X,Y\in\Gamma\left(TM\right)$,
\begin{align*}
(\nabla_X \phi ) Y &= g(X,Y+hY)\xi - \eta (Y)(X+hX),\\
(\nabla_X h ) Y &= ((1-\kappa)g(X, \phi Y)+g(X, \phi hY))\xi +
 \eta(Y)(h(\phi X + \phi h X)) - \mu \phi h Y.
\end{align*}
\end{theo}

Blair, Kouforgiorgos and Papantoniou proven also that the
$\left(\kappa,\mu\right)$-nullity condition remains unchanged
under $\cal D$-homothetic deformations. The concept of $\mathcal
D$-homothetic deformation for a contact metric manifold
$(M,\phi,\xi,\eta,g)$ has been introduced by S. Tanno in
\cite{tanno0} and then intensively studied by many authors. We
recall that, given a real positive number $a$, by a $\mathcal
D$-homothetic deformation of constant $a$ we mean a change of the
structure tensors in the following way:
\begin{equation}
\label{trans} \tilde \phi= \phi, \ \ \ \tilde \eta= a \eta, \ \ \
\tilde \xi= \frac 1a \xi, \ \ \ \tilde g= ag +a(a-1)  \eta \otimes
\eta.
\end{equation}
In \cite{blair1} the authors proven that if $M$ is a contact
metric manifold whose Reeb vector field  belongs to the $(\kappa,
\mu)$-nullity distibution then for the contact metric manifold
$(M,\tilde \phi,\tilde \xi,\tilde \eta,\tilde g)$ the same
property holds. Precisely $\tilde \xi$ belongs to the $(\tilde
\kappa ,\tilde \mu)$-nullity distribution where
$$
\tilde \kappa = \frac {\kappa + a^2 - 1}{a^2}, \ \ \ \ \tilde \mu =
\frac {\mu + 2a - 2} a.
$$

\subsection{Legendrian foliations}

A \emph{Legendrian distribution} on a contact manifold
$(M^{2n+1},\eta)$ is defined by an $n$-dimensional subbundle $L$
of the contact distribution such that $d\eta\left(X,X'\right)=0$
for all $X,X'\in\Gamma\left(L\right)$. When $L$ is integrable, it
defines a \emph{Legendrian foliation} of $(M^{2n+1},\eta)$.
Legendrian foliations have been extensively investigated in recent
years from various points of views (cf. \cite{pang},
\cite{libermann}, \cite{jayne1}, \cite{mino3}, etc.). In
particular Pang provided a classification of Legendrian foliations
by means of a bilinear symmetric form $\Pi_{\cal F}$ on the
tangent bundle of the foliation, defined by $\Pi_{\cal
F}\left(X,X'\right)=-\left({\cal L}_{X}{\cal
L}_{X'}\eta\right)\left(\xi\right)=-\eta\left(\left[X',\left[X,\xi\right]\right]\right)$.
He called a Legendrian foliation $\cal F$  \emph{non-degenerate},
\emph{degenerate} or \emph{flat} according to the circumstance
that the bilinear form $\Pi_{\cal F}$ is non-degenerate,
degenerate or vanishes identically, respectively. In terms of an
associated metric $g$, $\Pi_{\cal F}$ is given by
\begin{equation}\label{invariantep}
\Pi_{\cal F}\left(X,X'\right)=2g\left(\left[\xi,X\right],\phi X'\right).
\end{equation}
The last formula provides a geometrical interpretation of this classification:

\begin{lem}[\cite{jayne1}]\label{classificazione}
Let $\left(M,\phi,\xi,\eta,g\right)$ be a contact metric manifold
and let $\cal F$ be a foliation on it. Then
\begin{description}
    \item[(i)] $\cal F$ is flat if and only if
    $\left[\xi,X\right]\in\Gamma\left(T{\cal F}\right)$ for all
    $X\in\Gamma\left(T{\cal F}\right)$,
    \item[(ii)] $\cal F$ is degenerate if and only if there exist
    $X\in\Gamma\left(T{\cal F}\right)$ such that $\left[\xi,X\right]\in\Gamma\left(T{\cal F}\right)$,
    \item[(iii)] $\cal F$ is non-degenerate if and only if
    $\left[\xi,X\right]\notin \Gamma\left(T{\cal F}\right)$ for all
    $X\in\Gamma\left(T{\cal F}\right)$.
\end{description}
\end{lem}

Given a compatible contact metric structure
$\left(\phi,\xi,\eta,g\right)$ and a Legendrian distribution $L$
on $M$, we may consider the distribution $Q=\phi L$. It can be
proven (cf. \cite{jayne1}) that $Q$ is a Legendrian distribution
on $M$ which in general is not integrable, even if $L$ is; it is
called the \emph{conjugate Legendrian distribution} of $L$, and
the tangent bundle of $M$ splits as the orthogonal sum $TM=L\oplus
Q\oplus\mathbb{R}\xi$. When both $L$ and $Q$ are integrable, they
defines two orthogonal Legendrian foliations $\cal F$ and $\cal G$
on $M$, and the pair $\left(\cal F, \cal G\right)$ is an example
of a \emph{bi-Legendrian structure} on $M$. More in general a
bi-Legendrian structure is a pair of two complementary, not
necessarily orthogonal, Legendrian foliations on $M$.

In \cite{mino2} it has been attached to any contact manifold
$(M^{2n+1},\eta)$ endowed with a pair of two complementary
Legendrian distributions $\left(L,Q\right)$ a linear connection
$\bar{\nabla}$ uniquely determined by the following
properties:
\begin{align}\label{lista}
 \nonumber \textrm{(i)} \ \ &\bar{\nabla} L\subset L, \ \bar{\nabla} Q\subset Q, \ \bar{\nabla}\left(\mathbb{R}\xi\right)\subset\mathbb{R}\xi,\\
  \textrm{(ii)} \ \ &\bar{\nabla} d\eta=0,\\
 \nonumber \textrm{(iii)} \ \
  &\bar{T}\left(X,Y\right)=2d\eta\left(X,Y\right){\xi}, \textrm{ for all }
  X\in\Gamma\left(L\right), Y\in\Gamma\left(Q\right),\\
 \nonumber &\bar{T}\left(V,\xi\right)=\left[\xi,V_L\right]_Q+\left[\xi,V_Q\right]_L,
  \textrm{ for all } V\in\Gamma\left(TM\right),
\end{align}
where $\bar{T}$ denotes the torsion tensor of $\bar{\nabla}$ and
$V_L$ and $V_Q$ the projections of $V$ onto the
subbundles $L$ and $Q$ of $TM$, respectively. Such a connection is called the
\emph{bi-Legendrian connection} associated to the pair
$\left(L,Q\right)$ and it is defined as follows (cf. \cite{mino2}). For all $V\in\Gamma\left(TM\right)$, $X\in\Gamma\left(L\right)$ and
$Y\in\Gamma\left(Q\right)$,
$\bar\nabla_{V}X:=H\left(V_{L},X\right)_{L}+\left[V_{Q},X\right]_{L}+\left[V_{\mathbb R\xi},X\right]_L$,
$\bar\nabla_{V}Y:=H\left(V_{Q},Y\right)_{Q}+\left[V_{L},Y\right]_{Q}+\left[V_{\mathbb R\xi},Y\right]_Q$
and $\bar\nabla\xi=0$,  where $H$ denotes the
operator such that, for all $V,W\in\Gamma\left(TM\right)$,
$H\left(V,W\right)$ is the unique section of $\cal{D}$ satisfying
$i_{H\left(V,W\right)}d\eta|_{\cal{D}}=\left({\cal{L}}_{V}i_{W}d\eta\right)|_{\cal{D}}$. Further properties of this connection are
collected in the following proposition.

\begin{prop}[\cite{mino2}]\label{proprieta}
Let $\left(M,\eta\right)$ be a contact manifold endowed with two
complementary Legendrian distributions $L$ and $Q$ and let
$\bar{\nabla}$ denote the corresponding bi-Legendrian connection.
Then the $1$-form $\eta$ and the vector field $\xi$ are
$\bar{\nabla}$-parallel and the complete expression of the torsion
tensor field is given by
$\bar{T}\left(X,X'\right)=-\left[X,X'\right]_Q$ for all
$X,X'\in\Gamma\left(L\right)$ and
$\bar{T}\left(Y,Y'\right)=-\left[Y,Y'\right]_L$  for all
$Y,Y'\in\Gamma\left(Q\right)$.
\end{prop}

Now consider a contact metric manifold
$\left(M,\phi,\xi,\eta,g\right)$ endowed with two complementary
Legendrian distributions $L$ and $Q$. The definition of the
corresponding bi-Legendrian connection does not involve the
compatible metric $g$, however it makes sense to find conditions
which ensure $\bar{\nabla}$ being a metric connection at least
when $Q$ is orthogonal to $L$. This problem has been solved in
\cite{mino5} where the author proves the following result.

\begin{prop}\label{metrica}
Let $\left(M,\phi,\xi,\eta,g\right)$ be a contact metric manifold
and $L$ be a Legendrian distribution on $M$. Let $Q=\phi L$ be the
conjugate Legendrian distribution of $L$ and $\bar{\nabla}$ the associated
bi-Legendrian connection. Then
the following statements are equivalent:
\begin{description}
    \item[(i)] $\bar{\nabla} g=0$;
    \item[(ii)] $\bar{\nabla}\phi=0$;
    \item[(iii)] $g$ is a bundle-like metric with respect
both to the distribution $L\oplus \mathbb{R}\xi$ and to $Q\oplus
\mathbb{R}\xi$;
    \item[(iv)] $\bar{\nabla}_{X}X'=\left(\phi\left[X,\phi
    X'\right]\right)_L$ for all $X,X'\in\Gamma\left(L\right)$, $\bar{\nabla}_{Y}Y'=\left(\phi\left[Y,\phi
    Y'\right]\right)_Q$ for all $Y,Y'\in\Gamma\left(Q\right)$ and
    the operator $h$ maps the subbundle $L$
    onto $L$ and the subbundle $Q$ onto $Q$.
\end{description}
Furthermore, assuming $L$ and $Q$ integrable, (i)--(iv) are
equivalent to the total geodesicity of the Legendrian foliations
defined by $L$ and $Q$
\end{prop}

By a \emph{bi-Legendrian manifold} we mean a contact manifold endowed with two transversal Legendrian foliations. In particular, in this paper we deal with contact metric manifolds foliated by two mutually orthogonal Legendrian foliations. With regard to this, it will be useful in the sequel to prove the following lemma, which states essentially that in a bi-Legendrian manifold the operator $h$ is deeply linked to the given bi-Legendrian structure. This is just the starting point of our work.

\begin{lem}\label{lemma0}
Let $\cal F$ and $\cal G$ two mutually orthogonal Legendrian foliations on the contact metric manifold $\left(M,\phi,\xi,\eta,g\right)$. Then for all $X,X'\in\Gamma\left(T\cal F\right)$
\begin{equation}\label{equazioneinvarianti}
\Pi_{\cal F}\left(X,X'\right)-\Pi_{\cal G}\left(\phi X,\phi X'\right)=4g\left(hX,X'\right).
\end{equation}
\end{lem}
\begin{proof}
Since, by the orthogonality between $\cal F$ and $\cal G$ we have $\phi\left(T{\cal F}\right)=T{\cal G}$, using \eqref{invariantep} we have
\begin{align*}
\Pi_{\cal F}\left(X,X'\right)-\Pi_{\cal G}\left(\phi X,\phi X'\right)&=2g\left(\left[\xi,X\right],\phi X'\right)-2g\left(\left[\xi,\phi X\right],\phi^2 X'\right)\\
&=2g\left(\left[\xi,\phi X\right],X'\right)-2g\left(\phi\left[\xi,X\right],X'\right)\\
&=4g\left(hX,X'\right).
\end{align*}
\end{proof}

\begin{cor}\label{corollario1}
If $M$ is K-contact then $\cal F$ and $\cal G$ belong to the same
class according to the above Pang's classification.
\end{cor}

\begin{cor}\label{corollario2}
If $\cal F$ and $\cal G$ are both flat then $M$ is K-contact.
\end{cor}

\section{On the  bi-Legendrian  structure  associated to a contact metric $\left(\kappa,\mu\right)$-space}

Let $\left(M,\phi,\xi,\eta,g\right)$ be a contact metric manifold
such that $\xi$ belongs to the $\left(\kappa,\mu\right)$-nullity
distribution. By Theorem \ref{teoremagreci} the orthogonal
distributions $\mathcal{D}_{\lambda}$ and $\mathcal{D}_{-\lambda}$
defined by the eigenspaces of $h$ are involutive and define on
$M$ two orthogonal Legendrian foliations which we denote by $\mathcal{F}_{\lambda}$ and $\mathcal{F}_{-\lambda}$, respectively. In this section we begin the study of the bi-Legendrian manifold $\left(M,\mathcal{F}_{\lambda},\mathcal{F}_{-\lambda}\right)$.

\begin{prop}\label{flatness}
Let $\left(M,\phi,\xi,\eta,g\right)$ be a contact metric
$\left(\kappa,\mu\right)$-space which is not K-contact. Then the Legendrian foliations
$\mathcal{F}_{\lambda}$ and $\mathcal{F}_{-\lambda}$
are either non-degenerate or flat. More precisely,
$\mathcal{F}_{\lambda}$ (respectively, $\mathcal{F}_{-\lambda}$)
is flat if and only if $\kappa +\mu \lambda -\left( \lambda
+1\right) ^{2}=0$ (respectively, $\kappa -\mu \lambda -\left(
\lambda -1\right) ^{2}=0$), otherwise being non-degenerate.
\end{prop}
\begin{proof}
Let $X\in \Gamma \left(\mathcal{D}_{\lambda}\right)$. Then by
\eqref{definizione} we have
\begin{equation*}
R_{X \xi} \xi=\kappa X+\mu h X =\left( \kappa +\mu \lambda \right)
X.
\end{equation*}
On the other hand, using \eqref{Kcontatto},
\begin{align*}
R_{X \xi} \xi&=-\nabla _{\xi}\nabla _{X}\xi
-\nabla _{\left[ X,\xi\right] }\xi \\
&=\nabla _{\xi }\phi X+\lambda \nabla _{\xi } \phi X +\phi \left[
X,\xi\right] +\phi h \left[ X,\xi\right] \\
&=X-\lambda X-\left[ \phi X,\xi \right] +\lambda X-\lambda
^{2}X-\lambda \left[ \phi X,\xi \right] +\phi \left[ X,\xi%
\right]+\phi  h \left[ X,\xi \right]  \\
&=\left( \lambda +1\right) ^{2}X-\lambda \phi  \left[ X,\xi%
\right] +\phi  h \left[ X,\xi \right] ,
\end{align*}
so that%
\begin{equation*}
\phi  h \left[ X,\xi \right]   =\lambda \phi  \left[ X,\xi \right]
 +( \kappa +\mu \lambda -\left( \lambda +1\right)
^{2}) X
\end{equation*}
hence, applying $\phi$ and taking into account  that $\left[X,\xi\right]\in\Gamma\left(\cal D\right)$,
\begin{equation*}
-h\left[X,\xi\right]=-\lambda\left[X,\xi\right]+(\kappa+\mu\lambda-\left(\lambda+1\right)^2)\phi
X.
\end{equation*}
Decomposing $\left[X,\xi\right]$  in the directions of ${\cal D}_{\lambda}$ and ${\cal
D}_{-\lambda}$ we obtain
\begin{equation*}
-h(\left[X,\xi\right]_{{\cal
D}_{\lambda}}+\left[X,\xi\right]_{{\cal
D}_{-\lambda}})=-\lambda\left[X,\xi\right]+(\kappa+\mu\lambda-\left(\lambda+1\right)^2)\phi
X,
\end{equation*}
from which it follows that
\begin{equation}\label{pi0}
2\lambda\left[X,\xi\right]_{{\cal
D}_{-\lambda}}=(\kappa+\mu\lambda-\left(\lambda+1\right)^2)\phi
X
\end{equation}
and we conclude, according to Lemma \ref{classificazione}, that ${\cal
F}_{\lambda}$ is either flat or non-degenerate. The first case
occurs if and only if
$\kappa+\mu\lambda-\left(\lambda+1\right)^2=0$ and the second when
$\kappa+\mu\lambda-\left(\lambda+1\right)^2\neq 0$. In a similar
way one can prove the analogous results for ${\cal F}_{-\lambda}$.
\end{proof}

\begin{rem}
\emph{From Corollary \ref{corollario1} it follows that  the
bi-Legendrian structure $\left({\cal F}_{\lambda},{\cal
F}_{-\lambda}\right)$ is flat if and only if $\kappa=1$ and hence $M$
is Sasakian. This can be also prove in a direct way observing
that, according to Proposition \ref{flatness},  the functions
$f\left(\kappa,\mu\right)=\kappa+\mu\lambda-\lambda\left(\lambda+1\right)^2=2\left(\kappa-1\right)+\left(\mu-2\right)\sqrt{1-\kappa}$
and
$g\left(\kappa,\mu\right)=\kappa-\mu\lambda-\lambda\left(\lambda-1\right)^2=2\left(\kappa-1\right)+\left(2-\mu\right)\sqrt{1-\kappa}$
both vanish if and only if $\kappa=1$.}
\end{rem}

Proposition \ref{flatness} extends and improves the results
obtained in \cite{jayne2} for contact metric manifolds for which
$\xi$ belongs to the $\kappa$-nullity distribution (cf. \cite{tanno}),
i.e. the Levi Civita connection of $g$ satisfies $R_{V
W}\xi=\kappa\left(\eta\left(W\right)V-\eta\left(V\right)W\right)$. In
\cite{jayne2} the author proven that the bi-Legendrian structure
associated to such  contact metric manifolds is non-degenerate; we
recall that in his proof he used the fact that the non-degenerate
plane sections containing $\xi$ have constant sectional curvature
and this last property does not hold for contact metric
$\left(\kappa,\mu\right)$-spaces, as it is been proven in
\cite{blair1}.

We remark also that from the proof of Proposition \ref{flatness} it follows an explicit expression of the invariants $\Pi_{{\cal
F}_{\lambda}}$ and $\Pi_{{\cal
F}_{-\lambda}}$ of the Legendrian foliations ${\cal
F}_{\lambda}$ and ${\cal
F}_{-\lambda}$. More precisely, from \eqref{pi0} and \eqref{invariantep} one can prove the following proposition.

\begin{prop}
Let $\left(M,\phi,\xi,\eta,g\right)$ be a contact metric
$\left(\kappa,\mu\right)$-space which is not K-contact. Then
the canonical invariants associated
to the Legendrian foliations ${\cal F}_{\lambda}$ and
${\cal F}_{-\lambda}$ are given by
\begin{equation}\label{formuleinvarianti}
\Pi_{{\cal
F}_{\lambda}}=\frac{\left(\lambda+1\right)^2-\kappa-\mu\lambda}{\lambda}g|_{{\cal F}_{\lambda}\times{\cal F}_{\lambda}}
\  \emph{ and } \  \Pi_{{\cal
F}_{-\lambda}}=\frac{-\left(\lambda-1\right)^2+\kappa-\mu\lambda}{\lambda}g|_{{\cal F}_{-\lambda}\times{\cal F}_{-\lambda}},
\end{equation}
respectively.
\end{prop}

It should be  remarked that the pair $\left(\Pi_{{\cal
F}_{\lambda}},\Pi_{{\cal
F}_{-\lambda}}\right)$ is an invariant of the contact metric $\left(\kappa,\mu\right)$-space in question up to $\cal D$-homothetic deformations. Indeed let $(\tilde{\phi},\tilde{\xi},\tilde{\eta},\tilde{g})$ be a $\cal D$-homothetic deformation of $\left(\phi,\xi,\eta,g\right)$. Then first of all since $\tilde h =\frac{1}{2} \mathcal L_{\tilde \xi} \tilde \phi= \frac 1a h$ (cf.~\cite{blair1}), the eigenvalues of $\tilde h$
are $\pm \tilde \lambda =\pm \frac 1a \lambda$, apart from $0$. It follows that the eigenspaces $\mathcal D_{\tilde \lambda}$ and $\mathcal D_{- \tilde \lambda}$ coincide with $\mathcal D_{ \lambda}$ and $\mathcal D_{-\lambda}$
respectively. Next, for all $X,X'\in\Gamma\left(\cal D_{\tilde\lambda}\right)=\Gamma\left(\cal D_{\lambda}\right)$ we have $\Pi_{{\cal
F}_{\tilde\lambda}}\left(X,X'\right)=-\tilde\eta([X',[X,\tilde\xi]])=-a\eta\left(\frac{1}{a}\left[X',\left[X,\xi\right]\right]\right)=\Pi_{{\cal
F}_{\lambda}}\left(X,X'\right)$. Analogously one can prove that  $\Pi_{{\cal
F}_{-\tilde\lambda}}=\Pi_{{\cal
F}_{-\lambda}}$. Moreover, it should be observed that the invariant $\Pi_{\cal F}$ of any Legendrian foliation $\cal F$ depend only on the Legendrian foliation and on the contact form $\eta$ and not on the associated metric $g$.
In particular the function
\begin{equation}\label{invarianteboeckx}
\frac{\Pi_{\cal F_\lambda}\left(X,X'\right)+\Pi_{\cal F_{-\lambda}}\left(\phi X,\phi X'\right)}{\Pi_{\cal F_\lambda}\left(X,X'\right)-\Pi_{\cal F_{-\lambda}}\left(\phi X,\phi X'\right)},
\end{equation}
for all $X,X'\in\Gamma\left(\cal D_\lambda\right)$ such that $\Pi_{\cal F_\lambda}\left(X,X'\right)\neq 0$ (or, equivalently, $g\left(X,X'\right)\neq 0$), is an invariant of the bi-Legendrian manifold $M$ up to $\cal D$-homothetic deformations and it does not depend on the vector fields $X,X'\in\Gamma\left({\cal D}_{\lambda}\right)$. Indeed, after a straightforward computation, taking into account Lemma \ref{lemma0}, \eqref{formuleinvarianti} and \eqref{contattometrica}, one can find that \eqref{invarianteboeckx} is a constant and, more precisely, it is given by
\begin{equation*}
\frac{\Pi_{\cal F_\lambda}\left(X,X'\right)+\Pi_{\cal F_{-\lambda}}\left(\phi X,\phi X'\right)}{\Pi_{\cal F_\lambda}\left(X,X'\right)-\Pi_{\cal F_{-\lambda}}\left(\phi X,\phi X'\right)}=\frac{1-\frac{\mu}{2}}{4\sqrt{1-\kappa}}=\frac{1}4{I_M},
\end{equation*}
where $I_M$ is the invariant introduced by Boeckx in \cite{boeckx1} for classifying contact metric $\left(\kappa,\mu\right)$-spaces. In particular if ${\cal F}_\lambda$ (respectively ${\cal F}_{-\lambda}$) is flat then $I_M$ attains the value $4$ (respectively $-4$).  Moreover, we can also give an explicit formula for the constant $\mu$ in terms of Legedrian foliations
\begin{equation}
\mu=\frac{\Pi_{\cal F_\lambda}\left(X,X'\right)}{g\left(h X,X'\right)}=\frac{\Pi_{\cal F_\lambda}\left(X,X'\right)}{\lambda g\left(X,X'\right)}
\end{equation}
for all $X,X'\in\Gamma\left({\cal D_\lambda}\right)$ such that $g\left(X,X'\right)\neq 0$.

\section{An interpretation of contact metric $\left(\kappa,\mu\right)$-spaces}

Let $\left(M,\phi,\xi,\eta,g\right)$ be a contact metric
$\left(\kappa,\mu\right)$-space. We can attach to the
bi-Legendrian structure $\left({\cal F}_{\lambda},{\cal
F}_{-\lambda}\right)$ the corresponding bi-Legendrian connection
$\bar{\nabla}$, that is the unique linear connection on $M$ such
that \eqref{lista} hold. Furthermore we have the following result.

\begin{prop}\label{bilegendrian2}
Let $\left(M,\phi,\xi,\eta,g\right)$ be a contact metric
$\left(\kappa,\mu\right)$-space and let $\bar{\nabla}$ be the
bi-Legendrian connection associated to $M$. Then the tensors
$\phi$, $h$ and $g$ are $\bar{\nabla}$-parallel. Moreover, for the
torsion tensor of $\bar\nabla$ we have
$\bar{T}\left(Z,Z'\right)=2\Phi\left(Z,Z'\right){\xi}$ for  all
$Z,Z'\in\Gamma\left(\mathcal{D}\right)$.
\end{prop}
\begin{proof}
A well-known property about  ${\cal F}_{\lambda}$ and ${\cal
F}_{-\lambda}$ is that they are totally geodesic foliations (cf.
\cite{blair1}). Thus applying  Proposition \ref{metrica} we get
$\bar\nabla g=0$ and $\bar\nabla\phi=0$. Next, for all
$V\in\Gamma\left(TM\right)$,
$X\in\Gamma\left(\mathcal{D}_{+}\right)$,
$Y\in\Gamma\left(\mathcal{D}_{-}\right)$, we have
\begin{gather*}
\left(\bar{\nabla}_{V}h\right)X=\bar{\nabla}_V hX-h\bar{\nabla}_{V}X=\bar{\nabla}_V\left(\lambda
X\right)-\lambda\bar{\nabla}_{V}X=0,\\
\left(\bar{\nabla}_{V}h\right)Y=\bar{\nabla}_V hY-h\bar{\nabla}_{V}Y=\bar{\nabla}_V\left(-\lambda
Y\right)+\lambda\bar{\nabla}_{V}Y=0,
\end{gather*}
because $\bar\nabla$ preserves ${\cal F}_{\lambda}$
and ${\cal F}_{-\lambda}$.  Finally, for any $f\in C^{\infty}\left(M\right)$,
\begin{equation*}
\left(\bar{\nabla}_{V}h\right)f\xi=\bar{\nabla}_{V}\left(h\left(f\xi\right)\right)-h\left(\bar{\nabla}_{V}\left(f\xi\right)\right)=-h\left(f\bar{\nabla}_{V}\xi\right)-V\left(f\right)h\xi=0
\end{equation*}
because $\bar{\nabla}\xi=0$ and $h\xi=0$. It remains to prove the
property about the torsion, but it follows easily from Proposition
\ref{proprieta} and from the integrability of ${\cal D}_{\lambda}$ and
${\cal D}_{-\lambda}$.
\end{proof}

\begin{cor}\label{bilegendrian3}
With the assumptions and the notation of Proposition
\ref{bilegendrian2}, the connection $\bar\nabla$ is related to
the Levi Civita connection of $\left(M,\phi,\xi,\eta,g\right)$ by the following formula, for all
$X,Y\in\Gamma\left(\cal D\right)$,
\begin{equation}\label{levicivita}
\bar\nabla_{X}Y=\nabla_{X}Y-\eta\left(\nabla_{X}Y\right)\xi.
\end{equation}
\end{cor}
\begin{proof}
Since $\bar\nabla$ is torsion free along the leaves of the
foliations ${\cal F}_{\lambda}$ and ${\cal F}_{-\lambda}$ and, by
Proposition \ref{bilegendrian2}, it is metric, it coincides with
the Levi Civita connection along the leaves of ${\cal
F}_{\lambda}$ and ${\cal F}_{-\lambda}$. Hence \eqref{levicivita}
holds for all $X,Y\in\Gamma\left({\cal D}_{\lambda}\right)$ or
$X,Y\in\Gamma\left({\cal D}_{-\lambda}\right)$ because ${\cal
F}_{\lambda}$ and ${\cal F}_{-\lambda}$ are totally geodesic
foliations. Now let $X\in\Gamma\left({\cal D}_{\lambda}\right)$
and $Y\in\Gamma\left({\cal D}_{-\lambda}\right)$. It is well-known
(cf. \cite{blair1}) that $\nabla_{X}Y\in\Gamma\left({\cal
D}_{-\lambda}\oplus\mathbb{R}\xi\right)$. For all
$Y'\in\Gamma\left({\cal D}_{-\lambda}\right)$, using $\bar\nabla
g=0$, we have
\begin{align*}
2g(\nabla_{X}Y,Y')&=X(g(Y,Y'))+Y(g(X,Y'))-Y'(g(X,Y))+g([X,Y],Y')\\
&\quad+g([Y',X],Y)-g([Y,Y'],X)\\
&=X(g(Y,Y'))+g([X,Y],Y')+g([Y',X],Y)\\
&=X(g(Y,Y'))-g([X,Y']_{{\cal D}_{-\lambda}},Y)+g([X,Y]_{{\cal D}_{-\lambda}},Y')\\
&=2g([X,Y]_{{\cal D}_{-\lambda}},Y')\\
&=2g(\bar{\nabla}_{X}Y,Y'),
\end{align*}
from which it follows that $\bar{\nabla}_{X}Y=\left(\nabla_{X}Y\right)_{{\cal D}_{-\lambda}}$ and hence
\eqref{levicivita}. Analogously one can prove \eqref{levicivita}
for $X\in\Gamma\left({\cal D}_{-\lambda}\right)$ and
$Y\in\Gamma\left({\cal D}_{\lambda}\right)$.
\end{proof}

Now we examine in a certain sense an "inverse" problem. We start
with a bi-Legendrian structure on an arbitrary contact metric
manifold $M$ and we ask whether $M$ is a contact metric
$\left(\kappa,\mu\right)$-space for some
$\kappa,\mu\in\mathbb{R}$.

\begin{theo}\label{bilegendrian4}
Let $\left(M,\phi,\xi,\eta,g\right)$ be a contact metric manifold,
non $K$-contact, endowed with two orthogonal Legendrian foliations
${\cal F}$ and ${\cal G}$ and suppose that the bi-Legendrian
connection corresponding to $\left({\cal F},{\cal G}\right)$
satisfies $\bar{\nabla}\phi=0$ and $\bar{\nabla} h=0$. Then
$\left(M,\phi,\xi,\eta,g\right)$ is a contact metric
$\left(\kappa,\mu\right)$-space. Furthermore, the bi-Legendrian
structure $\left({\cal F}, {\cal G}\right)$ coincides with that
one determined by the eigenspaces of $h$.
\end{theo}
\begin{proof}
Firstly we prove that under our assumptions \eqref{levicivita}
holds. Since, by Proposition \ref{metrica}, $\bar{\nabla} g=0$ and $\bar{T}\left(X,X'\right)=0$,
$\bar{T}\left(Y,Y'\right)=0$ for all $X,X'\in\Gamma\left(T\cal
F\right)$ and $Y,Y'\in\Gamma\left(T\cal G\right)$, it follows
immediately that the bi-Legendrian connection and the Levi Civita
connection coincide along the leaves of $\cal F$ and $\cal G$.
Moreover, for all $X\in\Gamma\left(T\cal F\right)$ and
$Y\in\Gamma\left(T\cal G\right)$ $\nabla_{X}Y\in\Gamma\left(T{\cal
G}\oplus\mathbb{R}\xi\right)$ because for all
$X'\in\Gamma\left(T\cal F\right)$
\begin{equation*}
g\left(\nabla_{X}Y,X'\right)=X\left(g\left(Y,X'\right)\right)-g\left(Y,\nabla_{X}X'\right)=0
\end{equation*}
since $\cal F$, as well as $\cal G$, is totally geodesic by
Proposition \ref{metrica}. Then one can argue as in the
proof of Corollary \ref{bilegendrian3} and prove that
\begin{equation}\label{levicivita1}
\nabla_{Z}Z'=\bar{\nabla}_{Z}Z'+\eta\left(\nabla_{Z}Z'\right)\xi
\end{equation}
for all $Z,Z'\in\Gamma\left(\cal D\right)$. Now for all
$X,Y,Z\in\Gamma\left(\cal D\right)$ we have, applying \eqref{levicivita1},
\begin{align*}
g\left(\left(\nabla_{X}h\right)Y,Z\right)&=g\left(\nabla_{X}hY-h\nabla_{X}Y,Z\right)\\
&=g\left(\bar{\nabla}_{X}hY+\eta\left(\nabla_{X}hY\right)\xi-h\bar{\nabla}_{X}Y-\eta\left(\nabla_{X}Y\right)h\xi,Z\right)\\
&=g\left(\left(\bar{\nabla}_{X}h\right)Y,Z\right)+\eta\left(\nabla_{X}hY\right)\eta\left(Z\right)\\
&=g\left(\left(\bar{\nabla}_{X}h\right)Y,Z\right)=0,
\end{align*}
since, by assumption, $\bar{\nabla}h=0$. Thus the tensor field $h$
is $\eta$-parallel and so, by \cite[Theorem 4]{boeckx2},
$\left(M,\phi,\xi,\eta,g\right)$ is a contact metric
$\left(\kappa,\mu\right)$-space. For proving the last part of the
theorem, suppose by absurd that $\cal F$ does not coincide with
both ${\cal F}_{\lambda}$ and ${\cal F}_{-\lambda}$. Let $X$ be a
vector field tangent to $\cal F$ and decompose it as $X=X_+ +
X_-$, with $X_+\in\Gamma\left({\cal D}_{\lambda}\right)$ and
$X_-\in\Gamma\left({\cal D}_{-\lambda}\right)$. Then we have
$hX=h\left(X_+\right) + h\left( X_-\right)=\lambda X_+ - \lambda
X_- = \lambda\left(X_+ - X_-\right)$, from which, since by
Proposition \ref{metrica} $h$ preserves $\cal F$, it follows that
$X_+ - X_- \in \Gamma\left(T\cal F\right)$. On the other hand also
$X_+ + X_- = X \in \Gamma\left(T\cal F\right)$, hence $X_+$ and
$X_-$ are both tangent to $\cal F$ and this is a contradiction.
\end{proof}

From Theorem \ref{bilegendrian4} we get the following
characterization of contact metric $\left(\kappa,
\mu\right)$-spaces. Here, by an abuse of language, we call
Legendrian distribution of an almost contact manifold any
$n$-dimensional subbundle $L$ of the distribution $\cal
D=\ker\left(\eta\right)$ such that $d\eta\left(X,X'\right)=0$ for
all $X,X'\in\Gamma\left(L\right)$ and, as in contact metric
geometry, $2h$ is defined as the Lie differentiation of the tensor
$\phi$ along the Reeb vector field $\xi$.

\begin{theo}\label{main}
Let $\left(M,\phi,\xi,\eta,g\right)$ be an  almost  contact  metric
manifold with $\xi$ non-Killing. Then
$\left(M,\phi,\xi,\eta,g\right)$ is a contact metric $\left(\kappa,
\mu\right)$-space if and only if it admits two orthogonal conjugate
Legendrian distributions $L$ and $Q$ and a linear connection
$\tilde\nabla$ satisfying the following properties:
\begin{description}
    \item[(i)] $\tilde{\nabla}L\subset L$, $\tilde{\nabla}Q\subset Q$,
    \item[(ii)] $\tilde{\nabla}\eta=0$, $\tilde{\nabla}d\eta=0$, $\tilde{\nabla}g=0$, $\tilde{\nabla}h=0$,
    \item[(iii)]
    $\tilde{T}\left(Z,Z'\right)=2\Phi\left(Z,Z'\right){\xi}$ \
    for  all $Z,Z'\in\Gamma\left(\mathcal{D}\right)$,\\
    $\tilde{T}\left(V,\xi\right)=\left[\xi,V_{L}\right]_{Q}+\left[\xi,V_{Q}\right]_{L}$ \
    for all $V\in\Gamma\left(TM\right)$,
\end{description}
where $\tilde T$ denotes the torsion tensor field of
$\tilde\nabla$. Furthermore $\tilde\nabla$ is uniquely determined,
$L$ and $Q$ are integrable and coincide with the eigenspaces of
the operator $h$.
\end{theo}
\begin{proof}
The proof is rather obvious in one direction, it is sufficient to
take as $\tilde{\nabla}$ the bi-Legendrian connection associated
to the bi-Legendrian structure defined by the eigenspaces of $h$.
Now we prove the converse. Note that by (ii) it follows also that
$\xi$ is parallel with respect to $\tilde\nabla$, because for any
$V\in\Gamma\left(TM\right)$
$(\tilde\nabla_V\eta)\xi=-\eta(\tilde\nabla_V\xi)=0$, so
$\tilde\nabla_V\xi\in\Gamma\left(\cal D\right)$. On the other hand
for any $Z\in\Gamma\left(\cal D\right)$, since $\tilde\nabla$ is a
metric connection and preserves the subbundle ${\cal D} = L \oplus
Q$, we have
\begin{equation*}
g(\tilde\nabla_V\xi,Z)=V\left(g\left(\xi,Z\right)\right)-g(\xi,\tilde\nabla_VZ)=0,
\end{equation*}
from which $\tilde\nabla_V\xi$ is also orthogonal to $\cal D$ hence vanishes. Now we can prove the result. We show first that $d\eta=\Phi$, so $M$ is a contact metric manifold. For any $X,X'\in\Gamma\left(L\right)$ and $Y,Y'\in\Gamma\left(Q\right)$ we have $d\eta\left(X,X'\right)=0=g\left(X,\phi X'\right)$ and $d\eta\left(Y,Y'\right)=0=g\left(Y,\phi Y'\right)$. Moreover
\begin{equation*}
2\Phi\left(X,Y\right)\xi=\tilde{T}\left(X,Y\right)=\tilde\nabla_XY-\tilde\nabla_YX-\left[X,Y\right]
\end{equation*}
from which
\begin{equation}\label{equationeuno}
2\Phi\left(X,Y\right)=g(\tilde\nabla_XY,\xi)-g(\tilde\nabla_YX,\xi)-g\left(\left[X,Y\right],\xi\right).
\end{equation}
Now,
$g(\tilde\nabla_XY,\xi)=X\left(g\left(Y,\xi\right)\right)-g(Y,\tilde\nabla_X\xi)=0$
and, analogously, $g(\tilde\nabla_YX,\xi)=0$, so that
\eqref{equationeuno} becomes
\begin{equation*}
2\Phi\left(X,Y\right)=-\eta\left(\left[X,Y\right]\right),
\end{equation*}
from which it follows that $d\eta\left(X,Y\right)=\Phi\left(X,Y\right)$. For concluding that $\left(M,\phi,\xi,\eta,g\right)$ is a contact metric manifold it remains to check that $d\eta\left(Z,\xi\right)=\Phi\left(Z,\xi\right)$ for any $Z\in\Gamma\left(\cal D\right)$. Indeed $d\eta\left(Z,\xi\right)=-\frac{1}{2}\eta\left(\left[Z,\xi\right]\right)=0=\Phi\left(Z,\xi\right)$ since
\begin{equation*}
\left[Z,\xi\right]=\tilde\nabla_Z\xi-\tilde\nabla_\xi Z-\tilde T\left(Z,\xi\right)=-\tilde\nabla_\xi Z-\left[\xi,Z_L\right]_Q-\left[\xi,Z_Q\right]_L\in\Gamma\left(\cal D\right)
\end{equation*}
because of (i). Therefore $\left(M,\phi,\xi,\eta,g\right)$ is a contact metric manifold endowed with two complementary (in particular orthogonal) Legendrian distributions $L$ and $Q$, and since $\tilde\nabla\xi=0$ the connection $\tilde\nabla$ coincides with the bi-Legendrian connection $\bar\nabla$ associated to $\left(L,Q\right)$. This fact and (iii) imply the integrability of $L$ and $Q$. Indeed for any $X,X'\in\Gamma\left(L\right)$ we have
\begin{equation*}
[X,X']_Q=-\bar T(X,X')=-\tilde T(X,X')=-2d\eta(X,X')\xi=0
\end{equation*}
and
\begin{equation*}
g([X,X'],\xi)=\eta([X,X'])=-2d\eta(X,X')=0,
\end{equation*}
hence $\left[X,X'\right]\in\Gamma\left(L\right)$, and in a similar manner one can prove the integrability of $Q$. Thus $L$ and $Q$ define two orthogonal Legendrian foliations on $M$ and now the result follows from Theorem \ref{bilegendrian4}.
\end{proof}

The connection $\tilde\nabla$ is, under certain points of view, an "invariant" of the contact metric $\left(\kappa,\mu\right)$-space unless $\cal D$-homothetic deformations. Indeed, by a direct computation, one has the following result.

\begin{prop}
The bi-Legendrian connection associated to a contact metric $\left(\kappa,\mu\right)$-space remains unchanged under a $\cal D$-homothetic deformation.
\end{prop}

The connection stated in Theorem \ref{main} should be compared to the Tanaka-Webster connection of a non-degenerate integrable CR-manifold (cf. \cite{tanaka}, \cite{webster}) and to the generalized Tanaka-Webster connection introduced by Tanno in \cite{tanno2}. This can be seen in the following theorem, where we prove, using Theorem \ref{main}, the already quoted result that any contact metric $\left(\kappa,\mu\right)$-space is a strongly pseudo-convex CR-manifold.

\begin{cor}
Any contact metric $\left(\kappa,\mu\right)$-space is a
strongly pseudo-convex CR-manifold.
\end{cor}
\begin{proof}
We define a connection on $M$ as follows. We put
\begin{equation*}
\hat{\nabla}_{V}W=\left\{%
\begin{array}{ll}
   \bar\nabla_{V}W, & \hbox{ if $V\in\Gamma\left(\cal D\right)$;} \\
    -\phi h W+\left[\xi,W\right], & \hbox{ if $V=\xi$.} \\
\end{array}%
\right.
\end{equation*}
Then it easy to check that $\hat\nabla$ coincides with the
Tanaka-Webster connection of $M$  and so we get the assertion.
\end{proof}

The above characterization may be also a tool for proving properties on $\left(\kappa,\mu\right)$-spaces. As an application we show in a very simple way that an invariant submanifold of a contact metric metric $\left(\kappa,\mu\right)$-space, that is a submanifold $N$ such that $\phi T_pN \subset T_pN$ for all $p\in N$, is in turn a contact metric $\left(\kappa,\mu\right)$-space (cf. \cite{tripathi}).

\begin{cor}
Any invariant submanifold of a contact metric
$\left(\kappa,\mu\right)$-space is in turn a
$\left(\kappa,\mu\right)$-space.
\end{cor}
\begin{proof}
It is well-known (cf. \cite{blair0}) that an invariant submanifold of a contact metric manifold inherits a contact metric structure by restriction. Now let $N^{2m+1}$ be an invariant submanifold of $M^{2n+1}$ and
consider the distribution on $N$ given by $L_x:={T_{x}N}\cap{\cal
D}_{{\lambda x}}$ and $Q_x:={T_{x}N}\cap{\cal
D}_{{-\lambda x}}$ for all $x\in N$. It is easy to check
that $L$ and $Q$ define two mutually orthogonal Legendrian
foliations of $N^{2m+1}$ and that the bi-Legendrian connection
corresponding to $\left(L,Q\right)$ is just the connection induced
on $N$ by the bi-Legendrian connection associated to $\left({\cal
D}_{\lambda},{\cal D}_{-\lambda}\right)$. The
result now follows from Theorem \ref{main}.
\end{proof}

We conclude showing that the assumption in Theorem \ref{main} that
$\xi$ must be not Killing is essential. This can be seen in the
following example.

\begin{exam}
\emph{Consider the sphere
$S^3=\left\{\left(x_1,x_2,x_3,x_4\right)\in\mathbb{R}^4:x_1^2+x_2^2+x_3^2+x_4^2=1\right\}$ with the following Sasakian structure:}
\begin{equation*}
\eta=x_3dx_1+x_4dx_2-x_1dx_3-x_2dx_4, \
\xi=x_3\frac{\partial}{\partial x_1}+x_4\frac{\partial}{\partial
x_2}-x_1\frac{\partial}{\partial x_3}-x_2\frac{\partial}{\partial
x_4},
\end{equation*}
\begin{equation*}
g=\left(%
\begin{array}{cccc}
  1 & 0 & 0 & 0 \\
  0 & 1 & 0 & 0 \\
  0 & 0 & 1 & 0 \\
  0 & 0 & 0 & 1 \\
\end{array}%
\right), \ \phi=\left(%
\begin{array}{cccc}
  0 & 0 & -1 & 0 \\
  0 & 0 & 0 & -1 \\
  1 & 0 & 0 & 0 \\
  0 & 1 & 0 & 0 \\
\end{array}%
\right).
\end{equation*}
\emph{Set $X:=x_2\frac{\partial}{\partial
x_1}-x_1\frac{\partial}{\partial x_2}-x_4\frac{\partial}{\partial
x_3}+x_3\frac{\partial}{\partial x_4}$ and $Y:=\phi
X=x_4\frac{\partial}{\partial x_1}-x_3\frac{\partial}{\partial
x_2}+x_2\frac{\partial}{\partial x_3}-x_1\frac{\partial}{\partial
x_4}$, and consider the $1$-dimensional distributions $L$ and $Q$
on $S^3$ generated by $X$ and $Y$, respectively. An easy
computation shows that $\left[X,\xi\right]=-2Y$,
$\left[Y,\xi\right]=2X$, $\left[X,Y\right]=2\xi$. Thus $L$ and $Q$
defines two non-degenerate, orthogonal Legendrian foliations on the Sasakian manifold
$(S^3,\phi,\xi,\eta,g)$. For the
bi-Legendrian connection corresponding to this bi-Legendrian
structure, we have, after a straightforward computation, $\bar\nabla_X
X=\bar\nabla_X Y=\bar\nabla_X \xi=0$ and  $\bar\nabla_Y X=\bar\nabla_Y Y=\bar\nabla_Y
\xi=0$. Therefore $\bar\nabla\phi=0$ and so, by Proposition \ref{metrica}, also $\bar\nabla g=0$. Moreover, as $\xi$ is Killing obviously $\bar\nabla h=0$.}
\end{exam}

\end{document}